\documentclass[10pt,letterpaper]{article}

\usepackage{graphicx}
\usepackage{amsmath}
\usepackage{amssymb}
\usepackage{amsthm}
\usepackage{authblk}

\usepackage[pdfborder={0 0 0},pagebackref=true,breaklinks=true,letterpaper=true,colorlinks,bookmarks=false]{hyperref}

\usepackage{verbatim}
\graphicspath{{images/}}

\newcommand\todo[1]{\textcolor{red}{#1}}
\renewcommand\todo[1]{}

\newcommand{\normtwo}[1] {\left\| #1 \right\|_2}
\newcommand{\normn}[1] {\left\| #1 \right\|_n}

\newtheorem{theorem}{Theorem}[section]
\theoremstyle{definition}
\newtheorem{definition}{Definition}[section]
\newtheorem{lemma}[theorem]{Lemma}

\begin{document}

\title{Lipschitz Continuity of Mahalanobis Distances and Bilinear Forms}

\author[1]{Emonet R\'emi\\
\url{remi.emonet@univ-st-etienne.fr}}
\author[1]{\\Zantedeschi Valentina \\
\url{valentina.zantedeschi@univ-st-etienne.fr}}
\author[1]{\\Sebban Marc\\
\url{marc.sebban@univ-st-etienne.fr}}

\affil[1]{
Univ Lyon, UJM-Saint-Etienne, CNRS, Institut d Optique Graduate School, Laboratoire Hubert Curien UMR 5516, F-42023, SAINT-ETIENNE, France
}

\maketitle

\begin{abstract}
    Many theoretical results in the machine learning domain stand only for functions that are Lipschitz continuous.
Lipschitz continuity is a strong form of continuity that linearly bounds the variations of a function.
In this paper, we derive tight Lipschitz constants for two families of metrics: Mahalanobis distances and bounded-space bilinear forms.
To our knowledge, this is the first time the Mahalanobis distance is formally proved to be Lipschitz continuous and that such tight Lipschitz constants are derived.

\end{abstract}

\section{Multi-variate Lipschitz continuity}

\theoremstyle{definition}
A function is said Lipschitz continuous if it takes similar values on points that are close. More precisely, the slope of the function is bounded by a constant that is independent of the choice of points.
This means that the variation of a function that is Lipschitz continuous within a certain interval is small. The Lipschitz continuity is a strong form of uniform continuity: for instance, a function that is Lipschitz continuous is also continuous, but the reverse is not necessarily true. Let's take the example of the square function: $x^2$ is continuous on $\mathbb{R}^m$ but it is not Lipschitz continuous (the slope of $x^2$ is not bounded).

We now consider the Lipschitz continuity for a function $f: \mathcal{X}^2 \subset \mathbb{R}^d \times \mathbb{R}^d \to \mathbb{R}$.
\begin{definition}{(Multi-variate Lipschitz continuity)}\label{def:lips}
A function $f: \mathcal{X}^2 \subset \mathbb{R}^d \times \mathbb{R}^d \to \mathbb{R}$ is said $k_n$-lipschitz w.r.t. the norm $ \normn{.} $ if $\forall (x_1, x_2, {x}'_1,{x}'_2) \in \mathcal{X}^4$:
\begin{equation}
    |f(x_1,x_2) - f({x}'_1,{x}'_2)| \leq k_n \normn{{x_1 \choose x_2} - {{x}'_1 \choose {x}'_2} } \:.
\end{equation}
\end{definition}

If $f$ is differentiable on $\mathcal{X}^2 \subset \mathbb{R}^d \times \mathbb{R}^d$ and $\mathcal{X}^2$ is a convex space, the best constant $k_n$, the characteristic Lipschitz coefficient, can be estimated considering the fact that
\begin{align} 
    k_n &= \sup_{x_1, x_2, {x}'_1,{x}'_2 \in \mathcal{X}} \Bigg(\frac{|f(x_1,x_2) - f({x}'_1,{x}'_2)|}{\normn{{x_1 \choose x_2} - {{x}'_1 \choose {x}'_2}}}\Bigg) = \nonumber \\
    &=\sup_{ x_1, x_2 \in \mathcal{X}} \normn{\nabla f(x_1,x_2)} \:.
\end{align}

\section{Derivation for Particular Functions}
\label{sec:lips}
    In this section, we analyze the Lipschitz continuity of two classic metric functions: the Mahalanobis distance and the bilinear form. These two functions are largely used in the field of Machine Learning, especially in Metric Learning. 

\subsection{Derivation for Mahalanobis-like Distances}
    We recall that the Mahalanobis distance of a pair $(x_1,x_2) \in \mathcal{X}^2$ can be written as $d_{M}(x_1,x_2) = \sqrt{(x_1-x_2)^T M (x_1-x_2)} $ where $M$ is some Positive Semi-Definite matrix, whose coefficients can be optimized.
    By Def. \ref{def:lips}, the function $d_{M}: \mathcal{X}^2 \to \mathbb{R}$ is $k_n$-lipschitz w.r.t. the norm $ \normn{.} $ if $\forall x_1, x_2 \in \mathcal{X}$, $ \normn{\nabla d_{M}(x_1,x_2)} $ can be bounded by a constant $k_n$, where 
    $$\nabla d_{M}(x_1,x_2) = {\frac{\partial d_{M}(x_1,x_2)}{\partial x_1} \choose \frac{\partial d_{M}(x_1,x_2)}{\partial x_2}}$$
    and, for this particular case:

    \begin{align}
        \frac{\partial d_{M}(x_1,x_2)}{\partial x_1} & = \frac{1}{2\sqrt{(x_1-x_2)^T M (x_1-x_2)}} \frac{\partial }{\partial x_1} \Big((x_1-x_2)^T M (x_1-x_2)\Big) \nonumber \\
        & = \frac{1}{2\sqrt{(x_1-x_2)^T M (x_1-x_2)}} \frac{\partial }{\partial x_1} \left(x_1^TMx_1-x_2^TMx_1-x_1^TMx_2+x_2^TMx_2\right) \label{eq:der} \\
        & = \frac{2Mx_1-Mx_2-Mx_2}{2\sqrt{(x_1-x_2)^T M (x_1-x_2)}} \label{eq:simp}\\ 
        & = \frac{Mx_1-Mx_2}{\sqrt{(x_1-x_2)^T M (x_1-x_2)}} = \frac{M(x_1-x_2)}{\sqrt{(x_1-x_2)^T M (x_1-x_2)}} \nonumber
    \end{align}
    and, in the same way:
    \begin{align}
        \frac{\partial d_{M}(x_1,x_2)}{\partial x_2} = \frac{M(x_2-x_1)}{\sqrt{(x_1-x_2)^T M (x_1-x_2)}} . \nonumber
    \end{align}
    In Eq. \ref{eq:der} and \ref{eq:simp} we made use of the symmetry of the matrix $M$.    
    \begin{lemma}
        The Mahalanobis distance $ d_{M}(x_1,x_2) $ is $k$-lipschitz w.r.t. the norm $ \normtwo{.} $, with $k = \sqrt{2}\left\|L\right\|_2$, where $M = L^T L$, with $L$ a lower triangular matrix.
    \end{lemma}

    \begin{proof}
        \begin{align}
            & \max_{ x_1, x_2 \in \mathcal{X}} \normtwo{\nabla d_{M}(x_1,x_2)} \nonumber \\
            & = \max_{ x_1, x_2 \in \mathcal{X}} \sqrt{ \normtwo{\frac{\partial d_{M}(x_1,x_2)}{\partial x_1}}^2 + \normtwo{\frac{\partial d_{M}(x_1,x_2)}{\partial x_2}}^2} \nonumber \\
            & = \max_{ x_1, x_2 \in \mathcal{X}} \sqrt{\normtwo{\frac{M(x_1-x_2)}{\sqrt{(x_1-x_2)^T M (x_1-x_2)}}}^2 + \normtwo{\frac{M(x_2-x_1)}{\sqrt{(x_1-x_2)^T M (x_1-x_2)}}}^2} \nonumber \\
            & = \max_{ x_1, x_2 \in \mathcal{X}} \sqrt{2\normtwo{\frac{M(x_1-x_2)}{\sqrt{(x_1-x_2)^T M (x_1-x_2)}}}^2} \nonumber \\
            & = \max_{ x_1, x_2 \in \mathcal{X}} \sqrt{2\normtwo{\frac{L^T L(x_1-x_2)}{\sqrt{(x_1-x_2)^T L^T L (x_1-x_2)}}}^2} \label{eq:cholesky} \\
            & = \max_{ x_1, x_2 \in \mathcal{X}} \sqrt{2\normtwo{\frac{L^T(L(x_1-x_2))}{\sqrt{(L(x_1-x_2))^T L(x_1-x_2)}}}^2} \nonumber \\
            & = \max_{ x_1, x_2 \in \mathcal{X}} \sqrt{2\normtwo{L^T\frac{(L(x_1-x_2))}{\normtwo{L(x_1-x_2)}}}^2} \nonumber \\
            & \leq \max_{ x_1, x_2 \in \mathcal{X}} \sqrt{2\normtwo{L^T}\normtwo{\frac{(L(x_1-x_2))}{\normtwo{L(x_1-x_2)}}}^2} \label{eq:cauchy}\\
            & \leq \sqrt{2}\normtwo{L^T} = k . \label{eq:norm}
        \end{align}

    In Eq. \ref{eq:cholesky} we applied the Cholesky decomposition $M = L^T L$, the bound in \ref{eq:cauchy} is due to the Cauchy-Schwarz inequality and in Eq. \ref{eq:norm} $\normtwo{\frac{(L(x_1-x_2))}{\normtwo{L(x_1-x_2)}}} = 1$ because it is a normalized vector.

    \end{proof}

\subsection{Derivation for Bilinear Forms}
    We recall that the bilinear form of a pair $(x_1,x_2)$ is computed as $ d_{M}(x_1,x_2) = x_1^T M x_2 $, where $M$ is a generic matrix that can be optimized. 
    Then:
    $$\frac{\partial d_{M}(x_1,x_2)}{\partial x_1}  = \frac{\partial }{\partial x_1} \left(x_1^T M x_2)\right) = Mx_2 $$
    
    $$\frac{\partial d_{M}(x_1,x_2)}{\partial x_2} = \frac{\partial }{\partial x_2} \left(x_1^T M x_2)\right) = M^Tx_1 . $$

    \begin{lemma}The bilinear similarity $ d_{M}(x_1,x_2) = x_1^T M x_2 $ is $k$-lipschitz w.r.t. the norm $ \normtwo{.} $, with $k = \sqrt{2} \normtwo{M} R$, when $\normtwo{x} \leq R$ $\forall x \in \mathcal{X}$.
    \end{lemma}
    \begin{proof}
        \begin{align}
            \max_{\forall x_1, x_2 \in U} \normtwo{\nabla d_{M}(x_1,x_2)} & = \max_{\forall x_1, x_2 \in U} \sqrt{ \normtwo{\frac{\partial d_{M}(x_1,x_2)}{\partial x_1}}^2 + \normtwo{\frac{\partial d_{M}(x_1,x_2)}{\partial x_2}}^2} \nonumber \\
            & = \max_{\forall x_1, x_2 \in U} \sqrt{\normtwo{Mx_2}^2 + \normtwo{M^Tx_1}^2} \nonumber \\
            & \leq \sqrt{2} \normtwo{M} R = k . \label{eq:max-r}
        \end{align}
    \end{proof}
\section{Conclusion}
\label{sec:concl}

In this paper, we recalled a method for proving the Lipschitz continuity and for finding a tight Lipschitz constant of multivariate differentiable functions.
Using this approach, we computed tight Lipschitz constants for two families of metrics that are heavily used, especially in metric learning.
We have shown that the Mahalanobis distance is Lipschitz continuous and has a constant of $\sqrt{2}\left\|L\right\|_2$ (where $L$ is the square root of the correlation matrix).
We have also shown that the bilinear form $xMy$ is Lipschitz continuous with a constant $\sqrt{2} \normtwo{M} R$ (when the space is bounded by $R$).

Many theoretical results in the machine learning domain rely on Lipschitz continuity and depend on the Lipschitz constants. For example, the generalization bounds obtained in the context of the uniform stability (see~\cite{bellet2015metric}) can be derived by constraining the studied functions to be Lipschitz continuous and the tightness of those bounds depends on the value of the Lipschitz constant.
The derivations from this paper have been originally developed to derive theoretical bounds for~\cite{zantedeschi2016c2lm}.
We believe these results can also be used to derive tighter theoretical bounds in other domains of machine learning.

{\small
\bibliographystyle{ieee}

}

\end{document}